\documentclass[11pt]{amsart}
\usepackage{xcolor}
\usepackage{amstext}
\usepackage{amsfonts}
\usepackage{amssymb}
\usepackage{amsbsy}
\usepackage{latexsym}
\usepackage[T1]{fontenc}
\usepackage{xy}
\usepackage{hhline}
\xyoption{all}
\newcounter{num}[section] %

\newenvironment{theo}
{\refstepcounter{num}%
\bigskip\noindent{\bf Theorem~\arabic{section}.\arabic{num}. }\it}
{\smallskip}

\newenvironment{prop}
{\refstepcounter{num}%
\bigskip\noindent{\bf Proposition~\arabic{section}.\arabic{num}. }\it}

\newenvironment{lemma}
{\refstepcounter{num}%
\bigskip\noindent{\bf Lemma~\arabic{section}.\arabic{num}. }\it}

\newenvironment{proof_of}[1]{\medskip\noindent{\it Proof #1}}
{$\Box$ \bigskip}

\newcommand{\si}{\sigma}
\newcommand{\al}{\alpha}
\newcommand{\be}{\beta}
\newcommand{\ga}{\gamma}

\newcommand{\de}{\delta}

\newcommand{\tr}{\mathop{\rm tr}}

\newcommand{\mdeg}{\mathop{\rm mdeg}}
   
\newcommand{\diag}{\mathop{\rm diag}}
\newcommand{\Char}{\mathop{\rm char}}

\newcommand{\FF}{{\mathbb{F}}}   
\newcommand{\NN}{{\mathbb{N}}}

\newcommand{\mylabel}[1]{}

\newcommand{\mycomment}[1]{}

\begin{document}
\title[Separating invariants for $2\times 2$ matrices]{Separating invariants for $2\times 2$ matrices}

\thanks{The second author was supported by RFBR 16-31-60111 (mol\_a\_dk), the first and the third authors were supported by RFBR 17-01-00258 and the Presidents Programme Support of Young Russian Scientists (MK-1378.2017.1).}

\author{Ivan Kaygorodov}
\address{Ivan Kaygorodov\\ 
Universidade Federal do ABC, Santo Andre, SP, Brazil}
\email{kaygorodov.ivan@gmail.com  (Ivan Kaygorodov))}

\author{Artem Lopatin}
%
\address{Artem Lopatin\\ 
Sobolev Institute of Mathematics, Omsk Branch, SB RAS, Omsk, Russia}
\email{dr.artem.lopatin@gmail.com (Artem Lopatin))}

\author{Yury Popov}
\address{Yury Popov\\ 
State University of Campinas, 651 Sergio Buarque de Holanda, 13083-859 Campinas, SP, Brazil}
\email{yuri.ppv@gmail.com (Yury Popov))}

\begin{abstract} 
A minimal separating set is found for the algebra of matrix invariants of several $2\times 2$ matrices over an infinite field of arbitrary characteristic.


\end{abstract}

\maketitle

\section{Introduction}\label{section_intro}

\subsection{Definitions} All vector spaces, algebras, and modules are over an infinite field $\FF$ of characteristic $\Char{\FF}\neq2$, unless otherwise stated.  By an algebra we always mean an associative algebra.

To define the algebras of matrix invariants,  we consider  the polynomial algebra 
$$R=R_{n,d}=\FF[x_{ij}(k)\,|\,1\leq i,j\leq n,\, 1\leq k\leq d]$$
together with $n\times n$ {\it generic} matrices
$$X_k=\left(\begin{array}{ccc}
x_{11}(k) & \cdots & x_{1n}(k)\\
\vdots & & \vdots \\
x_{n1}(k) & \cdots & x_{nn}(k)\\
\end{array}
\right).$$
Denote by $\si_t(A)$ the $t^{\rm th}$ coefficient of the characteristic polynomial $\chi_A$ of $A$. As an example, $\tr(A)=\si_1(A)$ and $\det(A)=\si_n(A)$. The action of the general linear group $GL(n)$ over $R$ is defined by the formula: $g\cdot x_{ij}(k) = (g^{-1} X_k g)_{ij}$, where $(A)_{ij}$ stands for the $(i,j)^{\rm th}$ entry of a matrix $A$. The  set of all elements of $R$ that are stable with the respect to the given action is called the algebra of {\it matrix invariants} $R^{GL(n)}$ and this algebra is generated by $\si_t(b)$, where $1\leq t\leq n$ and $b$ ranges over all monomials in the generic matrices matrices 
 $X_1,\ldots,X_d$ (see~\cite{Sibirskii_1968}, \cite{Procesi76}, \cite{Donkin92a}). Note that in characteristic zero case the algebra $R^{GL(n)}$ is generated by $\tr(b)$, where $b$ is as above. The ideal of relations between the generators of $R^{GL(n)}$ was described  in~\cite{Razmyslov74, Procesi76, Zubkov96}.
 
Denote by $H=M(n)\oplus \cdots\oplus M(n)$ the direct sum of $d$ copies of the space $M(n)$ of all matrices $n\times n$ over $\FF$. The elements of $R$ can be interpreted as polynomial functions from $H$ to $\FF$ as follows: $x_{ij}(k)$ sends $u=(A_1,\ldots,A_d)\in H$ to $(A_k)_{i,j}$. For a monomial $c\in R$ denote by $\deg{c}$ its {\it degree} and by $\mdeg{c}$ its {\it multidegree}, i.e., $\mdeg{c}=(t_1,\ldots,t_d)$, where $t_k$ is the total degree of the monomial $c$ in $x_{ij}(k)$, $1\leq i,j\leq n$, and $\deg{c}=t_1+\cdots+t_d$. Similarly we denote the degree and multidegree of a $\NN$-homogeneous ($\NN^d$-homogeneous, respectively) polynomial of $R$, where $\NN$ stands for non-negative integers. Since $\deg{\si_t(X_{i_1}\cdots X_{i_s})}=ts$, the algebra $R^{GL(n)}$ has $\NN$-grading by degrees and $\NN^d$-grading by multidegrees. 

In 2002 Derksen and Kemper~\cite{DerksenKemper_book} introduced the notion of separating invariants as a weaker concept than generating invariants. Given a subset $S$ of $R^{GL(n)}$, we say that elements $u,v$ of $H$ {\it are separated by $S$} if  exists an invariant $f\in S$ with $f(u)\neq f(v)$. If  $u,v\in H$ are separated by $R^{GL(n)}$, then we simply say that they {\it are separated}. A subset $S\subset R^{GL(n)}$ of the invariant ring is called {\it separating} if for any $u, v$ from $H$ that are separated we have that they are separated by $S$.  A subset $S\subset R^{GL(n)}$ is called {\it 0-separating} if for any $u\in H$ such that $u$ and $0$ are separated we have that $u$ and $0$ are separated by $S$.

The main result of this paper is a description of a minimal (by inclusion) separating set for the algebra of matrix $GL(2)$-invariants for any $d$.

\begin{theo}\label{theo-main}  
\mylabel{theo-main}
The following set is a minimal separating set for the algebra of matrix invariants $R^{GL(2)}$ for every $d\geq1$:
$$\begin{array}{cl}
\tr(X_i), \det(X_i),& 1\leq i\leq d,\\
\tr(X_i X_j), & 1\leq i<j\leq d,\\
\tr(X_i X_j X_k),&  1\leq i<j<k\leq d.\\
\end{array}
$$
\end{theo}%
\medskip

\subsection{The known results for matrix invariants}
\label{section-known}

A minimal generating set for the algebra of matrix invariants $R^{GL(2)}$ is known, namely:
$$\tr(X_i), \det(X_i), 1\leq i\leq d,\; 
\tr(X_{i_1}\cdots X_{i_k}),  1\leq i_1<\cdots<i_k\leq d,
$$
where $k=2,3$ in case $\Char{\FF}\neq 2$ and $k>0$ in case $\Char{\FF}=2$ (see~\cite{Procesi_1984,DKZ_2002}). It is easy to see that the set
$$\tr(X_i), \det(X_i), 1\leq i\leq d,\; \sum\limits_{i+j=k,\; i<j}\tr(X_i X_j), 3\leq k\leq 2d-1 $$
is a minimal (by inclusion) 0-separating set for $R^{GL(2)}$ (see also~\cite{Teranishi_1988,DKZ_2002}). By Hilbert Theorem, the algebra of invariants is a finitely generated module over the subalgebra generated by a $0$-separated set. This result can be applied to construct a {\it system of parameters} (i.e. an algebraically independent set such that the algebra of invariants is finitely generated module over it) of the algebra of invariants. As an example, for $R_{3,3}^{GL(3)}$ a minimal 
0-separating set is constructed, which is also a system of parameters (see~\cite{Lopatin_Sib}). Similar results are also known for~$R_{3,2}^{GL(3)}$ and $R_{4,2}^{GL(4)}$ (see~\cite{Teranishi_1986, Lopatin_Sib}).

\subsection{The known general results}
\label{section-known-general}

The algebra of matrix invariants is a partial case of more general construction of an algebra of invariants. Namely, consider a linear
algebraic group $G$ with a regular action over a finite dimensional vector space $V$. Extend this action to the action of $G$ over the coordinate ring $\FF[V]$ by the natural way: $(g\cdot f)(v)=f(g^{-1}v)$ for all $g\in G$, $f\in\FF[V]$ and $v\in V$. Then the algebra of invariants is the following set: $\FF[V]^G=\{f\in\FF[V]\,|\,g\cdot f=f\}$. It is well-known that there always exists a finite separating set (see~\cite{DerksenKemper_book}, Theorem 2.3.15).

In~\cite{Domokos_2007} Domokos established that for a reductive group $G$ and $G$-modules $V,W$ a separating set $S$ for $\FF[W\oplus V^{m}]^G$ can be obtained by the extension of any separating set $S_0$ for $\FF[W\oplus V^{m_0}]^G$, where $m_0=\dim{V} + 1\leq m$. Namely, this extension is defined as follows: a function $f\in\FF[W\oplus V^{m_0}]$ is send to $f\circ \pi_{r_1,\ldots,r_{m_0}}:W\oplus V^{m}\to \FF$, where $1\leq r_1<\cdots<r_{m_0}\leq m$ and $\pi_{r_1,\ldots,r_{m_0}} :W\oplus V^{m}\to W\oplus V^{m_0}$ is the projection map sending $(w,v_1,\ldots,v_m)$ to $(w,v_{r_1},\ldots,v_{r_m})$. Note that $m_0$ does not depend on $m$. A similar result is not valid for sets of generators for matrix invariants (see Section~\ref{section-known}).

For a linear algebraic group $G$ denote by $d_G\in \NN\cup\{+\infty\}$ a minimal constant such that for each $G$-module $V$ as above the invariants of $\FF[V]^G$ are separated by elements of degree less or equal to $d_G$. Kohls and Kraft~\cite{Kohls_Kraft_2010} proved that $d_G$ is finite if and only if the group $G$ is finite. Separating invariants for the finite groups were considered in~\cite{Domokos_2016, Dufresne_2009, Dufresne_Elmer_Kohls_2009, Dufresne_Elmer_Sezer_2013, Dufresne_Kohls_2013, Dufresne_2015, Elmer_Kohls_2012,  Elmer_Kohls_2014, Elmer_Kohls_2016, Kohls_Sezer_2013, Neusel_2009, Sezer_2009, Sezer_2011}.

\section{Notations}
\label{section_aux}

This section contains some trivial remarks. If for $A,B\in M(n)$ there exists $g\in GL(n)$ such that $gAg^{-1}=B$, then we write $A\sim B$. Denote by $E_{ij}$ the matrix such that the $(i,j)^{\rm th}$ entry is equal to one and the rest of entries are zeros. 
The diagonal matrix with elements $a_1,\ldots,a_n$ we denote by $\diag(a_1,\ldots,a_n)$. The proof of the next lemma is straightforward. 

\begin{lemma}\label{lemma-Jordan} 
\mylabel{lemma-Jordan}
Assume that the field $\FF$ is algebraically closed and $A_1,A_2\in M(2)$, where 
$A_2=\left(
\begin{array}{cc}
\al_2 & \be_2\\
\ga_2 & \de_2\\
\end{array}
\right)$. Then
\begin{enumerate}
\item[(a)] for 
$A_1=\diag(\al_1,\be_1)$ there exists $g\in GL(2)$ such that $g A_1 g^{-1} = \diag(\be_1,\al_1)$ and 
$$g A_2 g^{-1}=\left(
\begin{array}{cc}
\de_2 & \ga_2\\
\be_2 & \al_2\\
\end{array}
\right);$$

\item[(b)] if $\ga_2\neq0$ or $\al_2\neq \de_2$ and 
$$A_1=\left(
\begin{array}{cc}
\al_1 & 1\\
0 & \al_1\\
\end{array}
\right),$$ 
then there exists $g\in GL(2)$ such that $g A_1 g^{-1} = A_1$ and 
$$g A_2 g^{-1}=\left(
\begin{array}{cc}
c & 0\\
\ga_2 & d\\
\end{array}
\right)$$
for some $c,d$.
\end{enumerate}
\end{lemma}

\section{The case of three matrices}

Denote the set from the formulation of Theorem~\ref{theo-main} by $S(d)$. 

\begin{lemma}\label{lemma123}
\mylabel{lemma123}
If $d\leq 3$, then the set $S(d)$ is a minimal separating set for $R^{GL(2)}$.
\end{lemma}
\begin{proof}
Assume that $d\in\{1,2,3\}$. Since in this case $S(d)$ is a (minimal) generating set for $R^{GL(2)}$ (see Section~\ref{section-known}), we have that $S(d)$ is a separating set. It remains to show that $S(d)$ is minimal.

Assume that $d=1$. Then $\tr(X_1)$ does not separate matrices $\diag(1,-1), 0$ and $\det(X_1)$ does not separate matrices $\diag(1,0)$ and $0$.

Assume that $d=2$. Obviously, it is enough to prove that the set $S(2)$ without $\tr(X_1X_2)$ is not separating. Consider 
$u=(E_{12},E_{21})\in H$. Then $\tr(X_1X_2)$ separates $u$ and $0$, but the rest of elements of $S(2)$ do not separate them.

Assume that $d=3$.  Obviously, it is enough to prove that the set $S(3)$ without $\tr(X_1X_2X_3)$ is not separating. Consider 
$u=(E_{11},E_{21},E_{12})$ and $v=(E_{22},E_{21},E_{12})$ from $H$. Then $\tr(X_1X_2X_3)$ separates $u$ and $v$, but the rest of elements of $S(3)$ do not separate them. 
\end{proof}

Note that Section~\ref{section-known} implies that $S(1)$ and $S(2)$ are minimal as 0-separating sets but $S(3)$ is not a minimal as 0-separating set.

\section{The case of four matrices}

It is easy to see that if the assertion of Proposition~\ref{prop4} (see below) is valid over the algebraic closure of the field $\FF$,
then it is also valid over the field $\FF$. Therefore, in this section we assume that the field $\FF$ is algebraically closed.

\begin{prop}\label{prop4}
\mylabel{prop4}
Assume that $d=4$. Consider $u=(A_1,A_2,A_3,A_4)$ and $v=(B_1,B_2,B_3,B_4)$ from $H$ such that for every $f\in S(4)$ we have $f(u)=f(v)$. Then for $h=\tr(X_1\cdots X_4)$ we have $h(u)=h(v)$.
\end{prop}
\bigskip

We split the proof of the proposition into several lemmas. 
By the formulation of the proposition, 
\begin{enumerate}
\item[($T_i$)] $\tr(A_i)=\tr(B_i)$, $1\leq i\leq 4$;

\item[($D_i$)] $\det(A_i)=\det(B_i)$, $1\leq i\leq 4$;

\item[($T_{ij}$)] $\tr(A_i A_j)=\tr(B_i B_j)$, $1\leq i<j\leq 4$;

\item[($T_{ijk}$)] $\tr(A_i A_j A_k)=\tr(B_i B_j B_k)$, $1\leq i<j<k\leq 4$.
\end{enumerate}

We have to show that
\begin{enumerate}
\item[($Q$)] $\tr(A_1\cdots A_4) = \tr(B_1\cdots B_4)$.
\end{enumerate}
If ($T$) $f=h$ is one of the above equalities, then we write $T$ for $f-h$. As an example, $T_1=\tr(A_1)-\tr(B_1)$.

We denote the entries of the matrices $A_1,\ldots,A_4$ as follows:
$$A_1=\left(
\begin{array}{cc}
a_1 & a_2\\
a_3 & a_4\\
\end{array}
\right),\quad
A_2=\left(
\begin{array}{cc}
a_5 & a_6\\
a_7 & a_8\\
\end{array}
\right),\quad
A_3=\left(
\begin{array}{cc}
a_9 & a_{10}\\
a_{11} & a_{12}\\
\end{array}
\right),\quad
A_4=\left(
\begin{array}{cc}
a_{13} & a_{14}\\
a_{15} & a_{16}\\
\end{array}
\right).
$$
Similarly, substituting $a_i \to b_i$ for all $i$ we denote entries of the matrices $B_1,\ldots,B_4$ by $b_1,\ldots,b_{16}$.

\begin{remark}\label{remark-Jordan} \mylabel{remark-Jordan}  
(1) Since elements of $S(4)$ are invariants with the respect to the action of $GL(2)$ over $H$ diagonally by conjugation and the field is algebraically closed, we can assume that either 
$A_1=\diag(\al,\be)$ or  
$$A_1=\left(
\begin{array}{cc}
\ga & 1\\
0 & \ga\\
\end{array}
\right)$$
for some $\al,\beta,\ga$ from $\FF$. 
\medskip

(2) By part~(b) of Lemma~\ref{lemma-Jordan}, we can assume that in the second case of part~(1) we have that either $(A_2)_{12}=0$ or $(A_2)_{21}=0$ and $(A_2)_{11}=(A_2)_{22}$.  
\end{remark}

\begin{remark}\label{remark-generators} \mylabel{remark-generators}
Denote by $G(d)$ the minimal generating set from Section~\ref{section-known}; in particular, $G(3)=S(3)$. Consider $u, v\in H$ such that $u$ and $v$ are not separated by elements of $G(d)$. Then $u$ and $v$ are not separated by any invariant $f$ of degree $d$, since $f$ is a polynomial in elements of $G(d)$.
\end{remark}
\bigskip

Remark~\ref{remark-generators} immediately implies the next remark:

\begin{remark}\label{remark-generators2} \mylabel{remark-generators2} Assume that $d=4$ and $u,v\in H$ are not separated by elements of $S(4)$. Then $u,v$ are not separated by invariants $\tr(X_i X_j X_k)$ for any pairwise different $1\leq i,j,k\leq 4$. If we also have that $u,v$ are not separated by $\tr(X_1\cdots X_4)$, then they are not separated by the invariant $\tr(X_{\si(1)}\cdots X_{\si(4)})$ for any permutation $\si\in S_4$.
\end{remark}


\begin{lemma}\label{lemma1-prop4}
\mylabel{lemma1-prop4}
Assume that the condition of Proposition~\ref{prop4} holds and $A_1=0$. Then the equality ($Q$) holds.
\end{lemma}
\begin{proof} By part~(1) of Remark~\ref{remark-Jordan} we can assume that 
$a_7 = 0$. Similarly, we can assume that $b_3 = 0$. Considering equalities ($T_1$)--($T_4$) we obtain that 
$$ b_4 = -b_1,\quad  b_8 = a_5 + a_8 - b_5, $$
$$ b_{12} = a_{9} + a_{12} - b_{9},\quad b_{16} = a_{13} + a_{16} - b_{13},$$
respectively. Equality ($D_1$) implies that $b_1 = 0$. In the case of $b_2=0$, we have that ($Q$) holds. Then without loss of generality we can assume that 
$b_2\neq 0$. Equalities ($T_{12}$), ($T_{13}$) and ($T_{14}$), respectively, imply that $b_7=b_{11}=b_{15}=0$, respectively. Then ($Q$) holds.
\end{proof}

\begin{lemma}\label{lemma2-prop4}
\mylabel{lemma2-prop4}
Assume that the condition of Proposition~\ref{prop4} holds, $A_1$ is scalar and $B_1$ is diagonal. Then the equality ($Q$) holds.
\end{lemma}
\begin{proof} We have $A_1=\diag(a_1,a_1)$ and $B_1=\diag(b_1,b_4)$.
Considering equalities ($T_1$)--($T_4$) we obtain that 
$$ b_4 = 2 a_1 - b_1,\quad  b_8 = a_5 + a_8 - b_5, $$
$$ b_{12} = a_{9} + a_{12} - b_{9},\quad b_{16} = a_{13} + a_{16} - b_{13},$$
respectively. Hence ($D_1$) implies $b_1 = a_1$ and the matrix $B_1=A_1$ is scalar. The equality $Q=a_1 T_{234}$ concludes the proof. 
\end{proof}

\begin{lemma}\label{lemma3-prop4}
\mylabel{lemma3-prop4}
Assume that the condition of Proposition~\ref{prop4} holds and $A_1, B_1$ are diagonal. Then the equality ($Q$) holds.
\end{lemma}
\begin{proof}
Applying Lemma~\ref{lemma2-prop4} we can assume that $A_1$ and $B_1$ are not scalars. Denote $A_1=\diag(a_1,a_4)$, $B_1=\diag(b_1,b_4)$, where $a_1\neq a_4$ and $b_1\neq b_4$.
Equalities ($T_1$)--($T_4$) imply that 
$$ b_4 = a_1 + a_4 - b_1,\quad  b_8 = a_5 + a_8 - b_5, $$
$$ b_{12} = a_{9} + a_{12} - b_{9},\quad b_{16} = a_{13} + a_{16} - b_{13},$$
respectively. We consider several possibilities for entries of the matrices. 

\begin{enumerate}
\item[(1)] {\it Assume $b_1 = a_1$}. Since $a_1\neq a_4$, equalities ($T_{12}$), ($T_{13}$), ($T_{14}$) imply that 
$$b_5 = a_5,\quad b_9 = a_9\;\; \text{ and }\;\; b_{13} = a_{13},$$
respectively.

\begin{enumerate}
\item[(1.1)] {\it Let $b_7=0$}. 
It follows from equality ($D_2$) that $a_6=0$ or $a_7=0$.  Since $A_1$ is diagonal, we can apply part (a) of Lemma~\ref{lemma-Jordan} (see also part~(1) of Remark~\ref{remark-Jordan}\mycomment{1}{Here we use the reference to part~(1) of Remark~\ref{remark-Jordan} just to show how to apply part (a) of Lemma~\ref{lemma-Jordan}. Namely, all elements of $S(4)$ and $Q$ are invariants with the respect to the action of $GL(2)$ diagonally by conjugation over $H$, then we can apply Lemma~\ref{lemma-Jordan}})  to the pair $(A_1,A_2)$ and assume that $a_7=0$.
Since $a_1\neq a_4$ the following equality holds:
$$ Q = \frac{1}{a_1-a_4}\left((a_1 a_5 - a_4 a_8) T_{134} \frac{}{} - a_1 a_4 (a_5 - a_8) T_{34}\right) + a_{13} T_{123} + a_{12} T_{124}.  $$
Hence $(Q)$ holds.

\item[(1.2)] {\it Let $b_7\neq 0$}. Thus equalities ($T_{23}$), ($T_{24}$), ($D_2$), respectively, imply that 
$$\begin{array}{ccl}
b_{10} & = & (a_6 a_{11} + a_7 a_{10} - b_6 b_{11})/ b_7, \\
b_{14} & = &(a_6 a_{15} + a_7 a_{14} - b_6 b_{15})/b_7, \\
b_6  & = & a_6 a_7/b_7,\\
\end{array}$$
respectively. In case $a_6 = 0$ the equality
$$Q = a_4 T_{234} + a_5 T_{134} - a_4 a_5 T_{34}$$
completes the proof. Thus we assume that $a_6\neq 0$. Since $a_1\neq a_4$, equalities ($T_{123}$) and ($T_{124}$), respectively, imply that
$$a_{11} = a_7 b_{11}/b_7 \;\;\text{ and }\;\; a_{15} = a_7 b_{15} / b_7,$$
respectively. Thus ($Q$) holds.
\end{enumerate}

\item[(2)] {\it Assume $b_1\neq a_1$}. Then it follows from ($D_1$) that $b_1 = a_4$. 
Since $a_1\neq a_4$, we equalities ($T_{12}$), ($T_{13}$) and ($T_{14}$), respectively, imply that
$$ b_5 = a_8,\quad 
b_9 = a_{12} \;\;\text{ and }\;\;
b_{13} = a_{16},$$
respectively.

\begin{enumerate}
\item[(2.1)] {\it Let $b_7=0$}. It follows from equality ($D_2$) that $a_6=0$ or $a_7=0$.  Since $A_1$ is diagonal, we can apply part (a) of Lemma~\ref{lemma-Jordan} to the pair $(A_1,A_2)$ and assume that $a_6=0$. Thus
$$ Q = a_4 T_{234} + a_5 T_{134} - a_4 a_5 T_{34},$$ 
i.e., ($Q$) holds.

\item[(2.2)] {\it Let $b_7\neq0$}. Then equalities ($D_2$), ($T_{23}$) and ($T_{24}$), respectively, imply that
$$\begin{array}{ccl}
b_6 & = & a_6 a_7/b_7,\\
b_{10} & = & (a_6 a_{11} + a_7 a_{10}  - b_6 b_{11})/b_7,\\
b_{14} & = &(a_6 a_{15} + a_7 a_{14}  - b_6 b_{15})/ b_7,
\end{array}$$
respectively. 
If $a_7=0$, then he equality
$$Q = a_4 T_{234} + a_5 T_{134} - a_4 a_5 T_{34},$$
completes the proof. Thus we assume that $a_7\neq 0$. 
Since $a_1\neq a_4$, the following equalities follow from ($T_{123}$) and ($T_{124}$), respectively:
$$a_{10} = a_6 b_{11}/b_7\;\;\text{ and }\;\;a_{14} = a_{6} b_{15}/ b_7.$$
Thus the equality ($Q$) is valid.
\end{enumerate}
\end{enumerate}
\end{proof}

\begin{lemma}\label{lemma4-prop4}
\mylabel{lemma4-prop4}
Assume that the condition of Proposition~\ref{prop4} holds,  $$A_1=\left(
\begin{array}{cc}
a_1 & 1\\
0 & a_1\\
\end{array}
\right) \quad\text{ and }\quad 
B_1=\left(
\begin{array}{cc}
b_1 & 0\\
0 & b_4\\
\end{array}
\right).$$
Then the equality ($Q$) holds.
\end{lemma}
\begin{proof}
Applying part (b) of Lemma~\ref{lemma-Jordan} to the pair $(A_1,A_2)$ we can assume that we have one of the following two cases: 
\begin{enumerate}
\item[(a)] $a_6=0$;

\item[(b)] $a_6\neq 0$, $a_7=0$, $a_8=a_5$.
\end{enumerate}
Equalities ($T_1$)--($T_4$) imply that in  both case  
$$ b_4 = 2 a_1 - b_1,\quad  b_8 = a_5 + a_8 - b_5, $$
$$ b_{12} = a_9+ a_{12} - b_{9},\quad b_{16} = a_{13} + a_{16} - b_{13},$$
respectively. Therefore, the equality $b_1 = a_1$ follows from ($D_1$). Applying equalities ($T_{12}$), ($T_{13}$), ($T_{14}$) we obtain that $a_7=a_{11}=a_{15}=0$. In case (a) the matrix $A_2$ is diagonal and Lemma~\ref{lemma3-prop4} together with Remark~\ref{remark-generators2} concludes the proof. 

Assume that case (b) holds. Since $B_1$ is scalar, applying part~(1) of Remark~\ref{remark-Jordan} to $B_2$, we can assume that $b_7=0$. The same reasoning as above imply that the entries of ``new''{} matrices $B_2,B_3,B_4$ satisfy the same relations as the entries of ``old''{} matrices $B_2,B_3,B_4$. By equality ($D_2$) we have $b_5=a_5$. Hence 
$$Q= a_1 a_5 T_{34} + a_1 b_{13} T_{23} + a_1(a_9 + a_{12} - b_9)T_{24},$$
i.e., ($Q$) holds.
\end{proof}

\begin{lemma}\label{lemma5-prop4}
\mylabel{lemma5-prop4}
Assume that the condition of Proposition~\ref{prop4} holds for  $$A_1=\left(
\begin{array}{cc}
a_1 & 1\\
0 & a_1\\
\end{array}
\right) \quad\text{ and }\quad 
B_1=\left(
\begin{array}{cc}
b_1 & 1\\
0 & b_1\\
\end{array}
\right),$$
and $\Char{\FF}=2$. Then the equality ($Q$) holds.
\end{lemma}
\begin{proof}
By Lemma~\ref{lemma4-prop4} together with Remark~\ref{remark-generators2} we can assume that for any $i$ the matrix $A_i$ is not diagonal and for any $j$ the matrix $B_j$ is not diagonal.

Applying part (b) of Lemma~\ref{lemma-Jordan} to the pair $(A_1,A_2)$ we can assume that one of the following case holds:
\begin{enumerate}
\item[(1)] $a_6 = 0$,

\item[(2)] $a_6 \neq0$, $a_7 = 0$, $a_8 = a_5$.
\end{enumerate}
In both cases equalities ($T_2$)--($T_4$) imply that 
$$ b_8 = a_5 + a_8 - b_5, $$
$$ b_{12} = a_9 + a_{12} - b_{9},\quad b_{16} = a_{13} + a_{16} - b_{13},$$
respectively. Since $\Char{\FF} = 2$, then the equality $b_1 = a_1$ follows from ($D_1$). Applying equalities ($T_{12}$), ($T_{13}$) and ($T_{14}$), respectively, we obtain
$$ b_7 = a_7, \quad b_{11} = a_{11},\quad b_{15} = a_{15},$$
respectively.

Assume that we have case (1), i.e.,  $a_6 = 0$. 
Since the matrix $A_2$ is not diagonal, then $a_7\neq 0$. Then equalities  ($T_{23}$) and ($T_{24}$), respectively, imply that
$$\begin{array}{ccl}
b_{10} & = & \frac{1}{a_7}(a_5 b_9 + a_8 (b_9 - a_9) + a_9 b_5 - a_{11} b_6 + a_{12} (b_5 - a_5) - 2 b_5 b_9) + a_{10},\\
b_{14} &= &\frac{1}{a_7}(a_5 b_{13} + a_8 (b_{13} - a_{13}) + a_{13} b_5 - a_{15} b_6 + a_{16} (b_5 - a_5) - 2 b_{13} b_5) + a_{14},\\
\end{array}$$
respectively. Therefore, it follows from ($T_{123}$) and ($T_{124}$) that
$$\begin{array}{ccl}
b_9 & = & \frac{1}{a_7} a_{11} (b_5-a_5) + a_9,\\
b_{13} & = & \frac{1}{a_7} a_{15} (b_5-a_5)+a_{13},\\
\end{array}$$
respectively. If follows from equality ($D_2$) that
$b_6 = (a_5-b_5)(- a_8 + b_5))/a_7$. Thus equality ($Q$) is valid.

Assume that we have case (2). Equality ($D_2$) implies that $b_5=a_5$. 

Assume that $a_6\neq b_6$. Then equalities ($T_{23}$) and ($T_{24}$) imply that $a_{11}=0$ and $a_{15}=0$, respectively. The equality $Q=a_1 a_5 T_{34}$, i.e., ($Q$) is valid.

Finally, in case $a_6= b_6$ the equality 
$$Q = (a_5 + a_1 a_6)T_{134}  - a_1^2 a_6 T_{34}$$
concludes the proof.
\end{proof}

\begin{proof_of}{of Proposition~\ref{prop4}.}
If $\Char{\FF}$ is not two, then the fact that $S(4)$ is a generating set for $R^{GL(2)}$ concludes the proof. 

Assume that $\Char{\FF}=2$. By Lemma~\ref{lemma1-prop4}, we can assume that for every $i$ the matrix $A_i$ is non-zero as well as the matrix $B_i$. Lemmas~\ref{lemma3-prop4},~\ref{lemma4-prop4},~\ref{lemma5-prop4} together with part~(1) of Remark~\ref{remark-Jordan} conclude the proof.
\end{proof_of}


\section{The proof of Theorem~\ref{theo-main}}

In this section we complete the proof of Theorem~\ref{theo-main}.
Assume $d\geq4$. Since $S(3)$ is a minimal separating set in case $d=3$ (see Lemma~\ref{lemma123}), we have that $S(d)$ does not contain a proper subset which is separating. On the other hand, in case $\Char{\FF}\neq 2$ the set $S(d)$ generates the algebra $R^{GL(2)}$, thus $S(d)$ is separating. 

Assume that $\Char{\FF}=2$ and $u=(A_1,\ldots,A_d)$, $v=(B_1,\ldots,B_d)$ are not separated by $S(d)$. Proposition~\ref{prop4} together with the description of the generating set for $R^{GL(2)}$ from Section~\ref{section-known} implies that $u,v$ are not separated by any invariant of degree four. This fact allows us to apply Proposition~\ref{prop4} to $(A_1,A_2,A_3,A_4A_5)$ and $(B_1,B_2,B_3,B_4B_5)$ and obtain that $u,v$ are not separated by $\tr(X_1\cdots X_5)$. The description of the generating set for $R^{GL(2)}$ implies that $u,v$ are not separated by any invariant of degree five. Repeating this reasoning we obtain that $u,v$ are not separated by any invariant of degree $d$. Thus $u,v$ are not separated. Hence the set $S(d)$ is separating and the theorem is proven.

\end{document}